\documentclass[12pt]{article}
\usepackage[a4paper,margin=1.0in]{geometry}
\usepackage[colorlinks,citecolor=magenta,linkcolor=black]{hyperref}
\pdfpagewidth=\paperwidth \pdfpageheight=\paperheight
\usepackage{amsfonts,amssymb,amsthm,amsmath,eucal,tabu,url,longtable}
\usepackage{pgf}
 \usepackage{array}
 \usepackage{tikz-cd}
 \usepackage{pstricks}
 \usepackage{pstricks-add}
 \usepackage{pgf,tikz}
 \usetikzlibrary{automata}
 \usetikzlibrary{arrows}
 \usepackage{indentfirst}
\usepackage{tabularx} 
\usepackage{booktabs}
\usepackage{xcolor}

\theoremstyle{plain}
\newtheorem{thm}{Theorem}[section]
\newtheorem{theorem}[thm]{Theorem}

\newtheorem{proposition}[thm]{Proposition}

\theoremstyle{definition}
\newtheorem{definition}[thm]{Definition}
\newtheorem{remark}[thm]{Remark}

\newtheorem{thevarthm}[thm]{\varthmname}

\newenvironment{varthm*}[1]{\trivlist\item[]{\bf #1.}\it}{\endtrivlist}

\renewcommand\geq{\geqslant}

\renewcommand\leq{\leqslant}

\let\tilde=\widetilde

\newcommand\be{\begin{eqnarray*}}
\newcommand\ee{\end{eqnarray*}}

\newcommand\newop[2]{\def#1{\mathop{\rm #2}\nolimits}}
\newop\log{log}
\newop\ord{ord}
\newop\Gal{Gal}
\newop\SL{SL}
\newop\Bl{Bl}
\newop\mult{mult}
\newop\mass{mass}
\newop\div{div}
\newop\codim{codim}
\newop\sing{sing}
\newop\vdim{vdim}
\newop\edim{edim}
\newop\Ass{Ass}
\newop\size{size}
\newop\reg{reg}
\newop\satdeg{satdeg}
\newop\supp{supp}
\newop\Neg{Neg}
\newop\Nef{Nef}
\newop\Nefh{Nef_H}
\newop\Eff{Eff}
\newop\Zar{Zar}
\newop\MB{MB}
\newop\MBxC{MB\mathit{(x,C)}}
\newop\NnB{NnB}
\newop\Bigg{Big}
\newop\Effbar{\overline{\Eff}}
\newcommand\eqnref[1]{(\ref{#1})}

\def\keywordname{{\bfseries Keywords}}%
\def\keywords#1{\par\addvspace\medskipamount{\rightskip=0pt plus1cm
\def\and{\ifhmode\unskip\nobreak\fi\ $\cdot$
}\noindent\keywordname\enspace\ignorespaces#1\par}}
\def\subclassname{{\bfseries Mathematics Subject Classification
(2020)}\enspace}
\def\subclass#1{\par\addvspace\medskipamount{\rightskip=0pt plus1cm
\def\and{\ifhmode\unskip\nobreak\fi\ $\cdot$
}\noindent\subclassname\ignorespaces#1\par}}

\begin{document}
\title{On arrangements of smooth plane quartics and their bitangents}
\author{Marek Janasz, Piotr Pokora, Marcin Zieli\'nski}
\date{\today}
\maketitle

\thispagestyle{empty}
\begin{abstract}
In the present paper, we revisit the geometry of smooth plane quartics and their bitangents from several perspectives. First, we study in detail the weak combinatorics of arrangements of bitangents associated with highly symmetric quartic curves. We consider quartic curves from the point of view of the order of their automorphism groups, in order to establish a lower bound on the number of quadruple intersection points for arrangements of bitangents associated with smooth plane quartics, which are smooth members of Ciani's pencil. We then construct new examples of $3$-syzygy reduced plane curves using smooth plane quartics and their bitangents.
\subclass{14H50, 14C20, 32S22}
\end{abstract}

\section{Introduction}

In recent years, there has been a 
growing interest among researchers 
in the geometry of plane curve 
arrangements. The classical 
approach is to study line 
arrangements in the (complex) 
projective plane and their 
important applications in many 
areas of contemporary 
mathematics, including topology, 
commutative algebra, or 
combinatorics. On the other hand, 
the geometry of curve arrangements 
is not as well-established as in 
the case of line arrangements, 
mostly due to many difficulties 
that can arise. The first issue is 
the fact that singularities of 
curve arrangements are usually not 
quasi-homogeneous, which causes 
many perturbations, and it is 
usually difficult to define the 
right notion of the combinatorics 
associated with the arrangements 
which will have a universal 
meaning for many purposes. More 
recently, conic-line arrangements 
in the complex projective plane 
have come into play, mostly in the 
context of freeness and nearly-freeness of curves, or monodromy 
groups, see for instance 
\cite{DimcaPokora,Malara,Pokora,PokSz,SchenckToh}. Here we would like to follow a new line of research devoted to the understanding of the geometry of curve arrangements, as indicated in \cite{addline,max}, by focusing on new tools and methods that will allow us to study combinatorial properties of free arrangements of curves with irreducible components of positive genus.
For example, Talar in \cite{talar} studies arrangements of smooth plane cubic curves and lines, and under some natural assumptions about the singularities of the arrangements, he provides a characterization of certain free cubic-line arrangements depending on their weak combinatorics. Here we want to go one step further and our main goal is to study the geometry and the freeness property for arrangements of smooth plane quartics and lines which admit some quasi-homogeneous singularities. 

Let us give a brief outline of the paper. Our first result is presented in Section $2$, where we give what can be called a Hirzebruch-type inequality for arrangements of smooth quartics and lines in the complex projective plane admitting certain quasi-homogeneous singularities, please consult Theorem \ref{hirc}. This result allows us to give the first rough estimate of the number of quadruple intersections for arrangements of bitangents, and may be useful, for example, for combinatorialists working on arrangements of bitangents and smooth plane quartics. Next results in Section $3$ are devoted to the description of incidences between the $28$ bitangents associated with very symmetric smooth plane quartics, such as the Klein quartic. The reason for us to study the combinatorics of such classical objects as arrangements of bitangents is the fact that we have not been able to find in the literature a complete description of the incidence structure of such line arrangements. It is known that the maximal multiplicity of the intersection points for arrangements of bitangents is four, see for instance \cite{salgado}, but except this fact we do not know much. Our choice of smooth quartics is determined by the fact that these curves may have large automorphism groups, suggesting the existence of points of high multiplicity in arrangements of bitangents. Using some classical results in Riemannian geometry, we prove that the number of quadruple points in arrangements of bitangents associated to smooth plane quartics is related to their involutions. Our result allows us to establish a lower-bound on the number quadruple points for arrangements of bitangents associated with smooth plane quartics which are smooth members of Ciani's pencil, namely we have at least $9$ such quadruple points, see Proposition \ref{ciani}. It is worth recalling here that Ciani's pencil consists of quartic forms invariant under the standard action of the group $(\mathbb{Z}_{2})^{2} \rtimes S_{3}$ on $\mathbb{C}^{3}$, where $S_{3}$ denotes the permutation group. We then focus on homological properties of the Milnor algebras associated with arrangements of lines and a smooth quartic. For instance, in Section $4$ we show that for degrees up to $6$ there are no free arrangements of quartics and lines with some prescribed singularities, see Proposition \ref{6l}. Then we give new families of examples of so-called $3$-syzygy curves with degree up to $12$. Our research in this context is motivated by some recent results due to Dimca and Sticlaru devoted to $3$-syzygy curves \cite{dimsti}.

Throughout this paper, we always work over the complex numbers. All symbolic computations were performed using \verb}SINGULAR} \cite{Singular}.

\section{Hirzebuch-type inequality for arrangements of smooth plane quartics and lines with quasi-homogeneous singularities}

The main aim of this section is to present a Hirzebruch-type inequality for quartic-line arrangements with some prescribed 
singularities. This result provides combinatorial constraints on the existence of quartic-line arrangements and might be of independent interest, for instance in combinatorics. We assume here that our arrangements of smooth quartics and lines in the complex projective plane admit only ${\rm ADE}$-singularities and the $X_{9}$ singularity. For the completeness of our presentation, we present below the local normal forms of the singularities we are interested in, and these equations are taken from Arnold's paper \cite{arnold}.
\begin{table}
\centering
\begin{tabular}{ll}
$A_{k}$ with $k\geq 1$ & $: \, x^{2}+y^{k+1}  = 0$, \\
$D_{k}$ with $k\geq 4$ & $: \, y^{2}x + x^{k-1}  = 0$,\\
$E_{6}$ & $: \, x^{3} + y^{4} = 0$, \\ \pagebreak
$E_{7}$ & $: \, x^{3} + xy^{3} = 0$, \\
$E_{8}$ & $:\, x^{3}+y^{5} = 0$, \\
$X_{9}$ with $a^{2} \neq 4$ & $:\, x^4 + y^4 +ax^{2}y^{2} = 0$.
\end{tabular}
\caption{Local normal forms.}
\label{aloc}
\end{table}

\newpage
Let us recall that $A_{1}$ singularites are commonly known as \textit{nodes}, $A_{3}$ singularites are called \textit{tacnodes}, $D_{4}$ singularities are called as \textit{ordinary triple points}, and $X_{9}$ singularities are called as \textit{ordinary quadruple points}.
\begin{theorem}
\label{hirc}
Let $\mathcal{QL} = \{\ell_{1}, \ldots , \ell_{d}, Q_{1}, \ldots , Q_{k}\} \subset \mathbb{P}^{2}_{\mathbb{C}}$ be an arrangement of $d\geq 1$ lines and $k\geq 1$ smooth quartics such that $4k+d\geq 6$. Assume that $\mathcal{QL}$ admits $n_{2}$ nodes, $t_{2}$ tacnodes, $t_{5}$ singularities of type $A_{5}$, $d_{6}$ singularities of type $D_{6}$, $t_{7}$ singularities of type $A_{7}$, $n_{3}$ ordinary triple and $n_{4}$ ordinary quadruple points, and no other singularities. Then one has
$$56k + n_{2} + \frac{3}{4}n_{3} \geq d + \frac{13}{8}d_{6} + \frac{5}{2}t_{2} + 5t_{5} + \frac{29}{4}t_{7}.$$
\end{theorem}

\begin{proof}
We present a short outline of the proof following the path explained in \cite{Pokora2,Pokora}. Let $D = \ell_{1} + \ldots + \ell_{d} + Q_{1} + \ldots + Q_{k}$ be our divisor with ${\rm deg}(D) := m = 4k + d$, where $k\geq 1$ and $d\geq 1$. We will work with the pair $\bigg(\mathbb{P}^{2}_{\mathbb{C}}, \frac{1}{2}D\bigg)$ which is effective, since $4k+d\geq 6$, and it is log-canonical since the following inequality holds:
\[\frac{1}{2} \leq \min\limits_{p \in {\rm Sing}(D)} \bigg\{{\rm lct}_{p}(D)\bigg\},\] where ${\rm lct}_{p}(D)$ denotes the log canonical threshold of $p \in {\rm Sing}(D)$. We are going to use Langer's inequality proved in \cite{Langer}, namely 
\begin{equation}
\label{logMY}
\sum_{p \in {\rm Sing}(D)}  3\bigg( \frac{1}{2}\bigg(\mu_{p} - 1\bigg) + 1 - e_{orb}\bigg(p,\mathbb{P}^{2}_{\mathbb{D}}, \frac{1}{2} D\bigg) \bigg) \leq \frac{5}{4}m^{2} - \frac{3}{2}m,
\end{equation}
where $\mu_{p}$ is the Milnor number of a singular point $p \in {\rm Sing}(D)$ and $e_{{\rm orb}}(\cdot)$ denotes the local orbifold Euler number. The left-hand side of the inequality \eqnref{logMY} is bounded from below by
$$\frac{9}{4}n_{2} + \frac{45}{8}t_{2}+\frac{117}{16}n_{3}  + \frac{35}{4}t_{5} + \frac{333}{32}d_{6} + \frac{189}{16}t_{7} + 15n_{4},$$
which is an easy computation, see for instance \cite[Theorem B]{Pokora} or \cite{Ziel}. Now we look at the right-hand side. 
First of all, observe that the following combinatorial count holds:
\begin{equation}
16\binom{k}{2} + 4kd + \binom{d}{2} = n_{2} + 2t_{2} + 3n_{3} + 3t_{5} + 4d_{6} + 4t_{7} + 6n_{4}.
\end{equation}
This gives us
$$5m^{2}-6m = 5\cdot(16k+d+2n_{2}+4t_{2}+6n_{3} + 6t_{5} + 8d_{6} + 8t_{7} +12n_{4})-6(4k+d).$$
Combining collected data above, we obtain
\begin{equation}
\label{hirzql}
56k + n_{2} + \frac{3}{4}n_{3}\geq d +  \frac{13}{8}d_{6} + \frac{5}{2} t_2 + 5t_{5} + \frac{29}{4}t_{7},
\end{equation}
which completes the proof.
\end{proof}

\section{Smooth plane quartics and their bitangents}

Let us consider the following pencil of plane quartics introduced by Ciani in \cite{Ciani}:
$$Q_{\lambda} \, : \quad x^4 + y^4 + z^4 + \lambda(x^2y^2 + x^2z^2 + y^2z^2),$$
where $\lambda \in \mathbb{C}$ and $x,y,z$ are the homogeneous coordinates in the complex projective plane.  There are certain particularly interesting curves being members of Ciani's pencil, namely
\begin{itemize}
    \item[a)] if $\lambda =0$, then we have the Dyck quartic curve that admits a group of $96$ collineations. Moreover, the Dyck curves admits exactly $12$ hyperflexes \cite{Dyck}. 
    \item[b)] if $\lambda$ is either root of $\lambda^2 + 3\lambda +18$, then we have the Klein quartic curve admitting a group of $168$ collineations, and we know that the Klein quartic curve is a unique Hurwitz curve (of genus $3$) having the largest possible group of collineations.
    \item[c)] if $\lambda=3$, then we have the Komiya--Kuribayashi quartic which admits exactly $12$ hyperflexes \cite{KK}. Moreover, this quartic admits a group of $24$ collineations that is isomorphic to $S_{4}$.
\end{itemize}

It is classically known that a 
general complex plane quartic 
curve has $28$ bitangents, and 
this result is attributed to J. 
Pl\"ucker \cite{Plucker}. In this 
section we are going to describe 
explicitly the incidences between 
bitangents for three particularly 
symmetric quartics introduced 
above, namely for the Klein 
quartic, the Dyck quartic, and the 
Komiya--Kuribayashi quartic. Let us 
point out here directly that for 
us a bitangent line is a line that 
is tangent to a given curve at two 
points. However, \textbf{we do not assume} that these two points are 
\textbf{distinct}, i.e. we allow them to collide, and in this situation we call the associated line as hyperosculating. In each case the 
associated line arrangement 
possesses only double and 
quadruple intersection points, 
which is an interesting phenomenon 
that we were not able to detect 
directly in the literature. Due to 
this reason, for the completeness 
of our investigations, we deliver, 
in each case, the incidence table 
for bitangents and quadruple 
intersections.
\begin{proposition}
\label{Kl}
The $28$ bitangents to the Klein quartic curve intersect along $252$ double and $21$ quadruple points.   
\end{proposition}
\begin{proof}
The property that $28$ bitangents intersect along $21$ quadruple points is well-known, see for example \cite[Exercise 6.22]{dolgbok}, but we still need to detect other intersections.
Let $Q$ be the defining equation of the arrangement of $28$ bitangents (we denote this arrangement by $\mathcal{L}$), and denote by $J_{Q} = \langle \frac{\partial Q}{\partial_{x}}, \frac{\partial Q}{\partial_{y}}, \frac{\partial Q}{\partial_{z}} \rangle$ the Jacobian ideal. Recall that
$${\rm deg}(J_{Q}) = 
\tau(\mathcal{L}) = \sum_{p \in {\rm Sing}(\mathcal{L})} (\rm{mult}_{p}(\mathcal{L})-1)^2,$$
where $\tau(\mathcal{L})$ denotes the total Tjurina number of $\mathcal{L}$, and the most right-hand side equality follows from 
the fact that all singularities 
are ordinary and quasi-homogeneous. Using \verb}SINGULAR}, we can check that ${\rm deg}(J_{Q}) = 441$. Now we 
are going to show that there are 
no triple intersections.
Observe that $441 = \tau(\mathcal{L}) = n_{2} + 4n_{3} + 9n_{4}$, $\binom{28}{2} = n_{2} + 3n_{3} + 6n_{4}$, and we obtain 
$$63 = n_{3} + 3n_{4}.$$
Since $n_{4}=21$, thus $n_{3}=0$. Finally, using one of the counts presented above, we get $n_{2}=252$, and this completes the proof.
\end{proof}

Next, we detect intersection points for the $28$ bitangents to the Dyck quartic curve. 
\begin{proposition}
\label{Dyck}
The $28$ bitangents to the Dyck quartic curve intersect along $288$ double and $15$ quadruple points.   
\end{proposition}
\begin{proof}
The proof goes analogously as in Proposition \ref{Kl}.
\end{proof}
Finally, we present numerical description for bitangents associated with the Komiya--Kuribayashi quartic.
\begin{proposition}
\label{Komiya}
The $28$ bitangents to the Komiya--Kuribayashi quartic curve intersect along $324$ double and $9$ quadruple points. 
\end{proposition}
\begin{proof}
The proof goes analogously as in Proposition \ref{Kl}.
\end{proof}

Now we would like to focus on the number of quadruple points in arrangements of bitangents. Using Theorem \ref{hirc}, we can prove the following.
\begin{proposition}
\label{bad}
Let $\mathcal{L} \subset \mathbb{P}^{2}_{\mathbb{C}}$ be an arrangement of $28$ bitangents associated with a smooth plane quartic $C$. Then the number of quadruple points $n_{4}$ in $\mathcal{L}$ is bounded from above by $44$.
\end{proposition}
\begin{proof}

We are going to use \eqref{hirzql} for the arrangement $\mathcal{QL}$ consisting of $C$ and the $28$ bitangents to $C$. Note that in this situation the arrangement $\mathcal{QL}$ can only have singularities of type $A_{1}$, $A_{3}$, $D_{4}$, $A_{7}$ and $X_{9}$. Denote by $t_{2}$ the number of tacnodes and by $t_{7}$ the number of $A_{7}$ singularities, we have
$$t_{2} = 32 + 2(12-h), \quad t_{7} = h$$
where $h$ is the number of hyperosculating lines to $C$.
Plugging this data into \eqref{hirzql}, we get
\begin{equation}
56 + n_{2} + \frac{3}{4}n_{3} \geq 28 + \frac{5}{2}(32+2(12-h)) + \frac{29}{4}h = 168 + \frac{9}{4}h \geq 168.
\end{equation}
This gives us
$$n_{2} + n_{3} \geq n_{2}+\frac{3}{4}n_{3} \geq 112.$$
Recall that we have the following naive count
$$n_{2}+3n_{3} + 6n_{4} = \binom{28}{2},$$
which gives us
$$\binom{28}{2} = 6n_{4} + 3n_{3} + n_{2} \geq 6n_{4} + 112.$$
Thus
$$n_{4} \leq \frac{1}{6}\bigg(378 - 112\bigg),$$
and finally we get $n_{4} \leq 44$.
\end{proof}
However, based on our experiments performed with smooth plane quartics with large automorphism groups (i.e., classically it means that the automorphism group has order at least $9$), we observed that the number of quadruple intersection points is at most $21$, and the maximal value is obtained for the $28$ bitangents to the Klein quartic. Now, using classical results from the geometry of Riemann surfaces, we show how to do much better and get a much better bound than the one obtained in Proposition \ref{bad}. Let us emphasize here that our results may be known to experts in classical algebraic geometry, but we have not been able to find them in full in the literature, and for this reason we have decided to give a comprehensive outline. It is worth noting here that the idea behind this part
was suggested to the second author by Igor Dolgachev in a private conversation. The first part of our outline is based on Dolgachev's notes \cite{Milano, Dolgachev}.

\begin{definition}
We say that a smooth complex projective curve $C$ is $\text{bielliptic}$ if it admits a degree $2$ cover of an elliptic curve.
\end{definition}
Let $C$ be a canonical curve of genus $3$ over $\mathbb{C}$ with a bielliptic involution $\tau : C \rightarrow C$, i.e., this is an involution such that the genus of the quotient curve is $1$. In its canonical plane model, $\tau$ is induced by a projective involution $\tilde{\tau}$ whose set of fixed points consists of a point $x_{0}$ and a line $\ell_{0}$. The intersection $\ell_{0}\cap C$ are the fixed points of $\tau$ on $C$. 
\begin{theorem}[S. Kowalevskaya, \cite{Kov}] The point $x_{0}$ is the intersection point of four distinct bitangents of $C$. Conversely, if a plane quartic has four bitangents intersecting at a point $x_{0}$, then there exists a bielliptic involution $\tau$ of $C$ such that the projective involution $\tilde{\tau}$ has $x_{0}$ as its isolated fixed point.
\end{theorem}
The above theorem tells us that the quadruple intersection points of bitangents are in correspondence with bielliptic involutions. It is classically known that for smooth complex projective curves of the genus $3$ the maximum possible number of bielliptic involutions is $21$ and this value is achieved for the Klein quartic. This means that we have exactly $21$ quadruple intersection points for the $28$ bitangents to the Klein quartic. Based on these observations, we can show the following general result.
\begin{proposition}
Let $C$ be a smooth complex 
quartic curve in
$\mathbb{P}^{2}_{\mathbb{C}}$ and 
let $\mathcal{L}$ be the 
associated arrangement of the $28$ 
bitangent lines. Then the number 
of quadruple intersection points 
of $\mathcal{L}$ is equal to the 
number of involutions of $C$. 
Moreover, the number of quadruple 
intersection points is less than 
or equal to $21$, and this upper-bound is achieved for the $28$ 
bitangents to the Klein quartic 
curve.
\end{proposition}
\begin{proof}
Recall that smooth plane quartics 
are never hyperelliptic. By a 
corollary to Accola's result 
\cite{Accola}, if a smooth complex 
projective curve $C$ of genus $3$ 
is not hyperelliptic, then every non-trivial
involution of $C$ is a bielliptic 
involution. This proves the first 
part of our result. For the second 
part, using Kowalevskaya's result, 
we see that indeed the $28$ 
bitangents to the Klein quartic 
deliver $21$ quadruple 
intersections and since the 
maximum possible number of 
bielliptic involutions is $21$ for 
smooth plane quartic curves, we 
get the sharp upper-bound.
\end{proof}
Let us now describe the geometry 
of our two remaining symmetric 
quartic curves via their 
automorphisms groups. We start 
with the Dyck quartic curve and 
here the automorphism group is 
$C_{4}^{2}\rtimes S_{3}$, which is 
of order $96$. This group is well-known in literature and described 
in \cite{Dyck,Edge}. First of all, 
Dyck noticed himself that the set 
of bitangent points is divided 
into two groups, namely we have 
$12$ hyperflexes (where the 
bitangent lines are 
hyperosculating) and ordinary 
tacnodes, so altogether there are 
$12 + 32$ tangential points. 
Observe that the first $12$ points 
are $A_{7}$ singularities viewed 
as intersections of 
hyperosculating tangent lines with 
the Dyck curve. Moreover, it was 
observed that these $12$ points 
can be divided into three groups 
of four elements - this can be 
done due to the existence of three 
special involutions. Each 
involution defines a line with the 
property that four of hyperflexes 
are lying on that line. We remind 
to the reader that these three 
lines are called perspective axes 
(Perspectivit\"ataxen). 
Furthermore, we can geometrically 
find other $12$ involutions, 
namely these are exactly the 
elements defining the $12$ lines 
in the $4$-th CEVA arrangement 
\cite{pokhir}. This implies that 
the $28$ bitangents to the Dyck 
curve deliver exactly $15$ 
quadruple intersection points, as 
we have already seen in 
Proposition \ref{Dyck}.

Let us take a look at the 
Komiya--Kuribayashi quartic. It is 
well-known that its automorphism 
group is $\Sigma_{4}$, the full 
symmetric group of $4$ elements. 
It is an easy exercise to check 
that in $\Sigma_{4}$ we have 
exactly $9$ non-trivial involutions, which 
means that the $28$ bitangents 
deliver exactly $9$ quadruple 
intersections, as we have seen in 
Proposition \ref{Komiya}.

Finishing discussion in this section, if we look at smooth members of Ciani's pencil. This pencil is invariant under the action of the group $\Sigma_{4}$ which can be viewed as the semi-direct product $(\mathbb{Z}/2\mathbb{Z})^{2}\rtimes \Sigma_{3}$ with $\Sigma_{3}$ acting by permutations on $x,y,z$ and $(\mathbb{Z}/2\mathbb{Z})^2$ acting by changing sign. This fact allows us to formulate the following main observation of this section.
\begin{proposition}
\label{ciani}
Let $C$ be a smooth complex plane quartic curve being a smooth member of Ciani's pencil, and let $\mathcal{L}$ be the arrangement of the $28$ bitangent lines to $C$. Then $\mathcal{L}$ has at least $9$ quadruple intersection points.
\end{proposition}
In the next subsections we focus on our particularly chosen smooth plane quartics and we deliver, in each case, a combinatorial description of the associated arrangement of bitangents. This description will be very useful for our further considerations devoted to homological properties of quartic-line arrangements. Our computations were implemented and performed in \verb}SINGULAR}.
\newpage
\subsection{The Klein quartic curve and its bitangents}
In this section $e$ is the number satisfying $e^{2}+e+2=0$.
\begin{table}[!htb]
\caption{$\:$  Equations of bitangents to the Klein quartic.}
\label{tab:TabLK1}
\begin{center} 
\begin{tabular}[b]{clclcl}
\toprule
$\ell_i$ & equation of line & $\ell_i$  &  equation of line & $\ell_i$  &  equation of line \\
\midrule
$\ell_1:$ & $y+ez$ &$\ell_{11}:$ & $x+y+(e-1)z $ &$\ell_{21}:$ & $x-ez$\\
$\ell_2:$& $y-ez$ &$\ell_{12}:$ &$x+ez $ &$\ell_{22}:$ &$x+y+(-e+1)z$\\
$\ell_3:$& $ex-y$ &$\ell_{13}:$ & $x+ey $ &$\ell_{23}:$ & $(e-1)x+y-z$\\
$\ell_4:$&$ey-z$ &$\ell_{14}:$ &$x-y+z  $ &$\ell_{24}:$ & $x+(-e+1)y-z$\\
$\ell_5:$&$(-e+1)x+y-z $ &$\ell_{15}:$ &$x+(e-1)y+z  $ &$\ell_{25}:$ & $x+(e-1)y-z$\\
$\ell_6:$&$x-y+(e-1)z $ &$\ell_{16}:$ &$ey+z  $ &$\ell_{26}:$ & $(e-1)x-y-z $\\
$\ell_7:$&$x+(1-e)y+z $ &$\ell_{17}:$ &$(e-1)x+y+z  $ &$\ell_{27}:$ & $x-y-z $\\
$\ell_8:$&$ex+y $ &$\ell_{18}:$ &$ex+z  $ &$\ell_{28}:$ & $ex-z$\\
$\ell_9:$&$x-ey $ &$\ell_{19}:$ &$x+y-z  $ &$ $ & $ $\\
$\ell_{10}:$&$x-y+(1-e)z $ &$\ell_{20}:$ &$x+y+z  $ &  & $ $\\
\bottomrule
\end{tabular}
\end{center}
\end{table}

\begin{table}[!htb]
\caption{$\:$  Quadruple intersection points.}
\label{tab:TabLK2}
\begin{center} 
\begin{tabular}[b]{clcl}
\toprule
$P_i$ & co-ordinates & $P_i$  &  co-ordinates \\
\midrule
$P_1:$ & $(1:0:0)$ & $P_{12}:$ &$(-1:0:1)$ \\
$P_2:$& $(e:-e-2:-e) $ & $P_{13}:$ &$(-2e-2:e-1:e-1) $ \\
$P_3:$& $(e:e+2:e)$ & $P_{14}:$ &$(e:-1:-1)$ \\
$P_4:$&$(-e:e+2:-e)$ &  $P_{15}:$ & $(-1:1:0)$\\
$P_5:$&$(e:e+2:-e)$ & $P_{16}:$ &$(0:1:0)$\\
$P_6:$&$(0:0:1)$ &$P_{17}:$ & $(e+2:e:-e) $\\
$P_7:$&$(-e+1:e-1:-2e-2)$ &$P_{18}:$ & $(-1:-1:e)$\\
$P_8:$ & $(-e+1:-e+1:2e+2)$ & $P_{19}:$ & $(-1:1:-e)$\\
$P_9:$ &$(0:1:1)$ & $P_{20}:$ & $(0:-1:1)$\\
$P_{10}:$ & $(-3e-2:e-2:-e+2)$ & $P_{21}:$ & $(1:0:1)$\\
$P_{11}:$ &$(1:1:0)$ & \\
\bottomrule
\end{tabular}
\end{center}
\end{table}
{\footnotesize
\begin{table}[]
\caption{Incidences between bitangents and quadruple intersection points.}
\rotatebox{-90}{
\begin{tabular}{@{}|l||l|l|l|l|l|l|l|l|l|l|l|l|l|l|l|l|l|l|l|l|l|@{}}
\hline
            & $P_1$ & $P_2$ & $P_3$ & $P_4$ & $P_5$ & $P_6$ & $P_7$ & $P_8$ & $P_9$ & $P_{10}$ & $P_{11}$ & $P_{12}$ & $P_{13}$& $P_{14}$ & $P_{15}$ & $P_{16}$ & $P_{17}$ & $P_{18}$ & $P_{19}$ & $P_{20}$ & $P_{21}$ \\ \hline\hline
$\ell_1$    & + & + & + &  &  &  &  &  &  &  &  &  &  &  &  &  &  &  &  &  &  \\ \hline
$\ell_2$    & + &  &  & + & + &  &  &  &  &  &  &  &  &  &  &  &  &  &  &  &  \\ \hline
$\ell_3$    &  & + &  & + &  & + &  &  &  &  &  &  &  &  &  &  &  &  &  &  &  \\ \hline
$\ell_4$    & + &  &  &  &  &  & + & + &  &  &  &  &  &  &  &  &  &  &  &  &  \\ \hline
$\ell_5$    &  &  &  & + &  &  & + &  & + &  &  &  &  &  &  &  &  &  &  &  &  \\ \hline
$\ell_6$    &  &  &  & + &  &  &  &  &  & + & + &  &  &  &  &  &  &  &  &  &  \\ \hline
$\ell_7$    &  &  &  &  &  &  & + &  &  & + &  & + &  &  &  &  &  &  &  &  &  \\ \hline
$\ell_8$    &  &  & + &  & + & + &  &  &  &  &  &  &  &  &  &  &  &  &  &  &  \\ \hline
$\ell_9$    &  &  &  &  &  & + &  &  &  & + &  &  & + &  &  &  &  &  &  &  &  \\ \hline
$\ell_{10}$ &  & + &  &  &  &  &  &  &  &  & + &  & + &  &  &  &  &  &  &  &  \\ \hline
$\ell_{11}$ &  &  & + &  &  &  &  &  &  &  &  &  &  & + & + &  &  &  &  &  &  \\ \hline
$\ell_{12}$ &  &  &  &  &  &  &  &  &  & + &  &  &  & + &  & + &  &  &  &  &  \\ \hline
$\ell_{13}$ &  &  &  &  &  & + &  &  &  &  &  &  &  & + &  &  & + &  &  &  &  \\ \hline
$\ell_{14}$ &  &  &  &  &  &  &  &  & + &  & + & + &  &  &  &  &  &  &  &  &  \\ \hline
$\ell_{15}$ &  &  &  &  &  &  &  &  &  &  &  & + &  & + &  &  &  & + &  &  &  \\ \hline
$\ell_{16}$ & + &  &  &  &  &  &  &  &  &  &  &  &  &  &  &  &  & + & + &  &  \\ \hline
$\ell_{17}$ &  &  & + &  &  &  &  &  &  &  &  &  &  &  &  &  &  & + &  & + &  \\ \hline
$\ell_{18}$ &  &  &  &  &  &  & + &  &  &  &  &  &  &  &  & + &  & + &  &  &  \\ \hline
$\ell_{19}$ &  &  &  &  &  &  &  &  & + &  &  &  &  &  & + &  &  &  &  &  & + \\ \hline
$\ell_{20}$ &  &  &  &  &  &  &  &  &  &  &  & + &  &  & + &  &  &  &  & + &  \\ \hline
$\ell_{21}$ &  &  &  &  &  &  &  &  &  &  &  &  & + &  &  & + & + &  &  &  &  \\ \hline
$\ell_{22}$ &  &  &  &  & + &  &  &  &  &  &  &  &  &  & + &  & + &  &  &  &  \\ \hline
$\ell_{23}$ &  &  &  &  & + &  &  & + & + &  &  &  &  &  &  &  &  &  &  &  &  \\ \hline
$\ell_{24}$ &  &  &  &  &  &  &  &  &  &  &  &  & + &  &  &  &  &  & + &  & + \\ \hline
$\ell_{25}$ &  &  &  &  &  &  &  & + &  &  &  &  &  &  &  &  & + &  &  &  & + \\ \hline
$\ell_{26}$ &  & + &  &  &  &  &  &  &  &  &  &  &  &  &  &  &  &  & + & + &  \\ \hline
$\ell_{27}$ &  &  &  &  &  &  &  &  &  &  & + &  &  &  &  &  &  &  &  & + & + \\ \hline
$\ell_{28}$ &  &  &  &  &  &  &  & + &  &  &  &  &  &  &  & + &  &  & + &  &  \\ \hline\hline
\end{tabular}
}
\end{table}
}
\subsection{The Fermat quartic curve and its bitangents}
In this section $\omega$ is the number satisfying $\omega^{4}+1=0$.

\begin{table}[!htb]
\caption{$\:$  Equations of bitangents to the Dyck quartic.}
\label{tab:TabLF1}
\begin{center} 
\begin{tabular}[b]{clclcl}
\toprule
$\ell_i$ & equation of line & $\ell_i$  &  equation of line & $\ell_i$  &  equation of line \\
\midrule
$\ell_1:$ & $-wy+z$ &$\ell_{11}:$ & $x - w^2y + z$ &$\ell_{21}:$ & $w^2x - w^2y + z$\\
$\ell_2:$& $wy+z$ &$\ell_{12}:$ &$x + w^2y + z$ &$\ell_{22}:$ &$w^2x + w^2y + z$\\
$\ell_3:$& $-w^3y + z$ &$\ell_{13}:$ & $-wx + z$ &$\ell_{23}:$ & $-w^3x + z$\\
$\ell_4:$&$w^3y + z$ &$\ell_{14}:$ &$wx + z$ &$\ell_{24}:$ & $w^3x + z$\\
$\ell_5:$&$-x - y + z $ &$\ell_{15}:$ &$-w^2x - y + z$ &$\ell_{25}:$ & $-wx + y$\\
$\ell_6:$&$-x + y + z$ &$\ell_{16}:$ &$-w^2x + y + z$ &$\ell_{26}:$ & $wx + y$\\
$\ell_7:$&$-x - w^2y + z$ &$\ell_{17}:$ &$-w^2x - w^2y + z$ &$\ell_{27}:$ & $-w^3x + y$\\
$\ell_8:$&$-x + w^2y + z$ &$\ell_{18}:$ &$-w^2x + w^2y + z$ &$\ell_{28}:$ & $w^3x + y$\\
$\ell_9:$&$x - y + z$ &$\ell_{19}:$ &$w^2x - y + z$ &$ $ & $ $\\
$\ell_{10}:$&$x + y + z$ &$\ell_{20}:$ &$w^2x + y + z$ &  & $ $\\
\bottomrule
\end{tabular}
\end{center}
\end{table}
\begin{table}[!htb]
\caption{$\:$  Quadruple intersection points.}
\label{tab:TabPF2}
\begin{center} 
\begin{tabular}[b]{clcl}
\toprule
$P_i$ & co-ordinates & $P_i$  &  co-ordinates \\
\midrule
$P_1:$ & $(1:0:0)$ & $P_{9}:$ &$(w^2:1:0)$ \\
$P_2:$& $(1:0:1) $ & $P_{10}:$ &$(0:1:-w^2) $ \\
$P_3:$& $(0:1:1)$ & $P_{11}:$ &$(-1:0:1)$ \\
$P_4:$&$(-1:1:0)$ &  $P_{12}:$ & $(0:1:0)$\\
$P_5:$&$(1:1:0)$ & $P_{13}:$ &$(-1:0:-w^2)$\\
$P_6:$&$(0:1:-1)$ &$P_{14}:$ & $(-1:0:w^2) $\\
$P_7:$&$(0:1:w^2)$ &$P_{15}:$ & $(0:0:1)$\\
$P_8:$ & $(-w^2:1:0)$ &  & \\
\bottomrule
\end{tabular}
\end{center}
\end{table}

{\footnotesize
\begin{table}[]
\begin{center}
\begin{tabular}{@{}|l||l|l|l|l|l|l|l|l|l|l|l|l|l|l|l|@{}}
\hline
            & $P_1$ & $P_2$ & $P_3$ & $P_4$ & $P_5$ & $P_6$ & $P_7$ & $P_8$ & $P_9$ & $P_{10}$ & $P_{11}$ & $P_{12}$ & $P_{13}$& $P_{14}$ & $P_{15}$  \\ \hline\hline
$\ell_1$    & + &  &  &  &  &  &  &  &  &  &  &  &  &  &    \\ \hline
$\ell_2$    & + &  &  &  &  &  &  &  &  &  &  &  &  &  &    \\ \hline
$\ell_3$    & + &  &  &  &  &  &  &  &  &  &  &  &  &  &    \\ \hline
$\ell_4$    & + &  &  &  &  &  &  &  &  &  &  &  &  &  &    \\ \hline
$\ell_5$    &  & + & + & + &  &  &  &  &  &  &  &  &  &  &    \\ \hline
$\ell_6$    &  & + &  &  & + & + &  &  &  &  &  &  &  &  &    \\ \hline
$\ell_7$    &  & + &  &  &  &  & + & + &  &  &  &  &  &  &    \\ \hline
$\ell_8$    &  & + &  &  &  &  &  &  & + & + &  &  &  &  &    \\ \hline
$\ell_9$    &  &  & + &  & + &  &  &  &  &  & + &  &  &  &    \\ \hline
$\ell_{10}$ &  &  &  & + &  & + &  &  &  &  & + &  &  &  &    \\ \hline
$\ell_{11}$ &  &  &  &  &  &  & + &  & + &  & + &  &  &  &    \\ \hline
$\ell_{12}$ &  &  &  &  &  &  &  & + &  & + & + &  &  &  &    \\ \hline
$\ell_{13}$ &  &  &  &  &  &  &  &  &  &  &  & + &  &  &    \\ \hline
$\ell_{14}$ &  &  &  &  &  &  &  &  &  &  &  & + &  &  &    \\ \hline
$\ell_{15}$ &  &  & + &  &  &  &  &  & + &  &  &  & + &  &    \\ \hline
$\ell_{16}$ &  &  &  &  &  & + &  & + &  &  &  &  & + &  &    \\ \hline
$\ell_{17}$ &  &  &  & + &  &  & + &  &  &  &  &  & + &  &    \\ \hline
$\ell_{18}$ &  &  &  &  & + &  &  &  &  & + &  &  & + &  &    \\ \hline
$\ell_{19}$ &  &  & + &  &  &  &  & + &  &  &  &  &  & + &    \\ \hline
$\ell_{20}$ &  &  &  &  &  & + &  &  & + &  &  &  &  & + &    \\ \hline
$\ell_{21}$ &  &  &  &  & + &  & + &  &  &  &  &  &  & + &    \\ \hline
$\ell_{22}$ &  &  &  & + &  &  &  &  &  & + &  &  &  & + &    \\ \hline
$\ell_{23}$ &  &  &  &  &  &  &  &  &  &  &  & + &  &  &    \\ \hline
$\ell_{24}$ &  &  &  &  &  &  &  &  &  &  &  & + &  &  &    \\ \hline
$\ell_{25}$ &  &  &  &  &  &  &  &  &  &  &  &  &  &  & +   \\ \hline
$\ell_{26}$ &  &  &  &  &  &  &  &  &  &  &  &  &  &  & +  \\ \hline
$\ell_{27}$ &  &  &  &  &  &  &  &  &  &  &  &  &  &  & +   \\ \hline
$\ell_{28}$ &  &  &  &  &  &  &  &  &  &  &  &  &  &  & +   \\ \hline\hline
\end{tabular}
\caption{Incidences between bitangents and quadruple intersection points.}
\end{center}
\end{table}
}
\subsection{The Komiya--Kuribayashi 
quartic curve and its bitangents}
In this section, we have $r=\sqrt{5}$ and $i$ is the imaginary unit.
\begin{table}[!htb]
\caption{$\:$  Equations of bitangents to the Komiya--Kuribayashi 
quartic.}
\label{tab:tab:TabLKK}
\begin{center} 
\begin{tabular}[b]{clclcl}
\toprule
$\ell_i$ & equation of line & $\ell_i$  &  equation of line & $\ell_i$  &  equation of line \\
\midrule
$\ell_1:$ & $-\frac{1}{5}r(2i+1)y+z$ &$\ell_{11}:$ & $x-2iy+z$ &$\ell_{21}:$ & $-\frac{1}{5}r(2i+1)x+z$\\
$\ell_2:$& $-\frac{1}{5}r(2i-1)y+z$ &$\ell_{12}:$ &$x+2iy+z$ &$\ell_{22}:$ &$-\frac{1}{5}r(2i-1)x+z$\\
$\ell_3:$& $\frac{1}{5}r(2i-1)y+z$ &$\ell_{13}:$ & $-2ix-y+z$ &$\ell_{23}:$ & $\frac{1}{5}r(2i-1)x+z$\\
$\ell_4:$&$\frac{1}{5}r(2i+1)y+z$ &$\ell_{14}:$ &$-2ix+y+z$ &$\ell_{24}:$ & $\frac{1}{5}r(2i+1)x+z$\\
$\ell_5:$&$-x-y+z$ &$\ell_{15}:$ &$-\frac{1}{2}ix-\frac{1}{2}iy+z$ &$\ell_{25}:$ & $-\frac{1}{5}r(2i+1)x+y$\\
$\ell_6:$&$-x+y+z$ &$\ell_{16}:$ &$-\frac{1}{2}ix+\frac{1}{2}iy+z$ &$\ell_{26}:$ & $-\frac{1}{5}r(2i-1)x+y$\\
$\ell_7:$&$-x-2iy+z$ &$\ell_{17}:$ &$\frac{1}{2}ix-\frac{1}{2}iy+z$ &$\ell_{27}:$ & $\frac{1}{5}r(2i-1)x+y$\\
$\ell_8:$&$-x+2iy+z$ &$\ell_{18}:$ &$\frac{1}{2}ix+\frac{1}{2}iy+z$ &$\ell_{28}:$ & $\frac{1}{5}r(2i+1)x+y$\\
$\ell_9:$&$x-y+z$ &$\ell_{19}:$ &$2ix-y+z$ &$ $ & $ $\\
$\ell_{10}:$&$x+y+z$ &$\ell_{20}:$ &$2ix+y+z$ &  & $ $\\
\bottomrule
\end{tabular}
\end{center}
\end{table}
\begin{table}[!htb]
\caption{$\:$  Quadruple intersection points.}
\label{tab:TabPKK}
\begin{center} 
\begin{tabular}[b]{clcl}
\toprule
$P_i$ & co-ordinates & $P_i$  &  co-ordinates \\
\midrule
$P_1:$ & $(1:0:-1)$ & $P_{6}:$ &$(0:1:1)$ \\
$P_2:$& $(1:0:0) $ & $P_{7}:$ &$(-1:1:0) $ \\
$P_3:$& $(1:0:1)$ & $P_{8}:$ &$(1:1:0)$ \\
$P_4:$&$(0:1:-1)$ &  $P_{9}:$ & $(0:0:1)$\\
$P_5:$&$(0:1:0)$ & &\\
\bottomrule
\end{tabular}
\end{center}
\end{table}
{\footnotesize
\begin{table}[]
\begin{center}
\begin{tabular}{@{}|l||l|l|l|l|l|l|l|l|l|@{}}
\hline
            & $P_1$ & $P_2$ & $P_3$ & $P_4$ & $P_5$ & $P_6$ & $P_7$ & $P_8$ & $P_9$  \\ \hline\hline
$\ell_1$    &  & + &  &  &  &  &  &  & \\ \hline
$\ell_2$    &  & + &  &  &  &  &  &  & \\ \hline
$\ell_3$    &  & + &  &  &  &  &  &  & \\ \hline
$\ell_4$    &  & + &  &  &  &  &  &  & \\ \hline
$\ell_5$    &  &  & + &  &  & + & + &  & \\ \hline
$\ell_6$    &  &  & + & + &  &  &  & + & \\ \hline
$\ell_7$    &  &  & + &  &  &  &  &  & \\ \hline
$\ell_8$    &  &  & + &  &  &  &  &  & \\ \hline
$\ell_9$    & + &  &  &  &  & + &  & + & \\ \hline
$\ell_{10}$ & + &  &  & + &  &  & + &  & \\ \hline
$\ell_{11}$ & + &  &  &  &  &  &  &  & \\ \hline
$\ell_{12}$ & + &  &  &  &  &  &  &  & \\ \hline
$\ell_{13}$ &  &  &  &  &  & + &  &  & \\ \hline
$\ell_{14}$ &  &  &  & + &  &  &  &  & \\ \hline
$\ell_{15}$ &  &  &  &  &  &  & + &  & \\ \hline
$\ell_{16}$ &  &  &  &  &  &  &  & + & \\ \hline
$\ell_{17}$ &  &  &  &  &  &  &  & + & \\ \hline
$\ell_{18}$ &  &  &  &  &  &  & + &  & \\ \hline
$\ell_{19}$ &  &  &  &  &  & + &  &  & \\ \hline
$\ell_{20}$ &  &  &  & + &  &  &  &  & \\ \hline
$\ell_{21}$ &  &  &  &  & + &  &  &  & \\ \hline
$\ell_{22}$ &  &  &  &  & + &  &  &  & \\ \hline
$\ell_{23}$ &  &  &  &  & + &  &  &  & \\ \hline
$\ell_{24}$ &  &  &  &  & + &  &  &  & \\ \hline
$\ell_{25}$ &  &  &  &  &  &  &  &  & +\\ \hline
$\ell_{26}$ &  &  &  &  &  &  &  &  & +\\ \hline
$\ell_{27}$ &  &  &  &  &  &  &  &  & +\\ \hline
$\ell_{28}$ &  &  &  &  &  &  &  &  & +\\ \hline\hline
\end{tabular}
\caption{Incidences between bitangents and quadruple intersection points.}
\end{center}
\end{table}

}

\newpage
\begin{remark}
As it was pointed out by X. Roulleau to the authors, the duals to $9$ quadruple intersection points for bitangents to the Komiya--Kuribayashi quartic form an arrangement of $9$ lines which is simplicial.
\end{remark}
\section{{\it 3}-syzygy curves constructed with plane quartics and lines}
This section can be considered as a prequel to \cite{addline} and we will see in a moment what is the reason standing behind this claim. Before that, we need an algebraic preparation. Let us denote by $S := \mathbb{C}[x,y,z]$ the coordinate ring of  $\mathbb{P}^{2}_{\mathbb{C}}$ and for a homogeneous polynomial $f \in S$ let us denote by $J_{f}$ the Jacobian ideal associated with $f$.

Let $C : f=0$ be a reduced curve in $\mathbb{P}^{2}_{\mathbb{C}}$ of degree $d$ defined by $f \in S$. Denote by $M(f) := S/ J_{f}$ the associated Milnor algebra. 
\begin{definition}
We say that a reduced plane curve $C$ is an $m$-syzygy curve when $M(f)$ has the following minimal graded free resolution:
$$0 \rightarrow \bigoplus_{i=1}^{m-2}S(-e_{i}) \rightarrow \bigoplus_{i=1}^{m}S(1-d - d_{i}) \rightarrow S^{3}(1-d)\rightarrow S \rightarrow M(f) \rightarrow 0$$
with $e_{1} \leq e_{2} \leq ... \leq e_{m-2}$ and $1\leq d_{1} \leq ... \leq d_{m}$.
\end{definition}
In the setting of the above definition, the minimal degree of the Jacobian relations among the partial derivatives of $f$ is defined to be ${\rm mdr}(f) := d_{1}$.

Among many examples of $m$-syzygy plane curves, we can distinguish the following important classes, and this can be done via the above homological description.  Here by $\tau(C)$ we denote the total Tjurina number of $C$.
\begin{definition}
We say that 
\begin{itemize}
\item $C$ is \textbf{free} if and only if $m=2$ and $d_{1}+d_{2}=d-1$. Moreover, \cite{duP} tells us that a reduced plane curve $C$ with ${\rm mdr}(f)\leq (d-1)/2$ is free if and only if
\begin{equation}
\label{duPles}
(d-1)^{2} - d_{1}(d-d_{1}-1) = \tau(C).
\end{equation}
\item $C$ is \textbf{nearly-free} if and only if $m=3$, $d_{1}+d_{2} = d$, $d_{2}=d_{3}$, and $e_{1}=d+d_{2}$. Moreover, by a result due to Dimca \cite{Dimca1}, we know that $C$ is nearly-free if and only if
\begin{equation}
(d-1)^{2}-d_{1}(d-d_{1}-1)=\tau(C)+1.
\end{equation}
\item $C$ is \textbf{plus-one generated} of level $d_{3}$ if $C$ is $3$-syzygy such that $d_{1}+d_{2}=d$ and $d_{3} > d_{2}$. 
\end{itemize}
\end{definition}

We will study homological properties of quartic-line arrangements in the plane such that they admit only quasi-homogeneous singularities, which are listed at the beginning of Section $2$ in Table \ref{aloc}. We start with the following observation.

\begin{proposition}
\label{6l}
There is no free arrangement consisting of one line $\ell$ and one smooth plane quartic curve $Q$.
\end{proposition}
\begin{proof}
Observe that we have the following possibilities as the intersections of one line and one smooth plane quartic:
\begin{itemize}
\item $Q.\ell = P_{1} + P_{2} + P_{3} + P_{4}$, i.e., we have $4$ nodes.
\item $Q.\ell = 2P_{1} + P_{2}+P_{3}$, i.e., we have one tacnode and two nodes.
\item $Q.\ell = 2P_{1} + 2P_{2}$, i.e., we have two tacnodes.
\item $Q.\ell = 3P_{1} + P_{2}$, i.e., we have one $A_{5}$ singularity and one node.
\item $Q.\ell = 4P$, so we have one singularity of type $A_{7}$.
\end{itemize}
Based on the above analysis, the maximal possible total Tjurina number is equal to $7$. On the other hand, such any arrangement would be free if the total Tjurina number is equal to either $13$ or $12$, which completes the proof.
\end{proof}
Here we can say even more, namely there is no nearly-free arrangement consisting of a smooth plane quartic and a line since in such a situation the total Tjurina number would be equal to either $11$ or $12$.

Now let us focus on the following situation when we have two lines and one smooth plane quartic curve.
\begin{proposition}
There is no free arrangement $C$ consisting of one smooth plane quartic and two lines admitting nodes, tacnodes, ordinary triple intersections, singularities of type $D_{6}$,  singularities of type $A_{5}$, and singularities of type $A_{7}$.
\end{proposition}
\begin{proof}
 First, we recall that for a reduced plane curve with only ${\rm ADE}$ singularities of even degree $2m$ the minimal degree of the Jacobian relations satisfies $d_{1}\geq m-1$ and this fact follows from \cite{max}. Assume now that $C$ is an arrangement satisfying the above list of assumptions that is free. Then $d_{1}=2$ and 
$$\tau(C) = d_{1}^2 - d_{1}(d-1)+(d-1)^2 = 19.$$
We have the following system of Diophantine equations:
$$\begin{cases} 
n_{2} + 2t_{2} + 3n_{3} +3t_{5} + 4d_{6} + 4t_{7} = 9 \\ 
n_{2} + 3t_{2} + 4n_{3} + 5t_{5} + 6d_{6} + 7t_{7} = 19,
\end{cases}
$$
where $n_{2}$ is the number of nodes, $n_{3}$ is the number of ordinary triple points, $d_{6}$ is the number of $D_{6}$ singularities, $t_{2}$ is the number of tacnodes, $t_{5}$ is the number of $A_{5}$ singularities, and $t_{7}$ is the number of $A_{7}$ singularities. Using any solver one can check that the above system does not have any non-negative integer solution, which completes the proof.
\end{proof}
In the next step we aim to construct new examples of plus-one generated arrangements consisting of three lines and one smooth plane quartic. In order to do so, we start with the Komiya--Kuribayashi quartic curve and lines.
\begin{proposition}
An arrangement consisting of the Komiya--Kuribayashi quartic, its two bitangent lines and two hyperosculating lines with the property that these $4$ lines intersect at one quadruple point is plus-one generated.
\end{proposition}
\begin{proof}
We only present the proof in one case because other cases follow the same lines. Denote by $K(x,y,z)$ the defining equation of the Komiya--Kuribayashi quartic.
Consider the arrangement $\mathcal{KL}$ defined by
$$Q(x,y,z) = (-x-y+z)(-x+y+z)(-x-2iy+z)(-x+2iy+z)\cdot K(x,y,z).$$
The arrangement $\mathcal{KL}$ has exactly one ordinary quadruple point, four tacnodes and two singularities of type $A_{7}$, so the total Tjurina number of the arrangement is equal to $\tau(\mathcal{KL}) = 14 + 12 + 9 = 35$. Using \verb}SINGULAR}, we can compute the minimal free resolution of the Milnor algebra which has the following form:
$$0\rightarrow S(-13) \rightarrow S(-12) \oplus S^{2}(-11) \rightarrow S^{3}(-7)\rightarrow S.$$
 Since $(d_{1},d_{2},d_{3}) = (4,4,5)$, $d_{1}+d_{2}=8$, and $d_{3}>d_{2}$, the arrangement is plus-one generated.
\end{proof}
Next, we pass to the Dyck quartic and three lines.
\begin{proposition}
An arrangement consisting of the Dyck quartic and its $3$ hyperosculating lines intersecting at one ordinary triple point is plus-one generated.
\end{proposition}
\begin{proof}
Similar to above, we present our justification for only one case of $12$ being possible, as others go along the same lines.
Consider the arrangement $\mathcal{DL}$ defined by
$$Q(x,y,z) = (ey+z)(-ey+z)(e^{3}y+z)\cdot (x^4 + y^4 + z^4 ),$$
where $e^4+1=0$.
The arrangement $\mathcal{DL}$ has exactly one ordinary triple point and three singularities of type $A_{7}$, so the total Tjurina number of the arrangement is equal to $\tau(\mathcal{DL}) = 21 + 4 = 25$. 
 Using \verb}SINGULAR} we can compute the minimal free resolution of the Milnor algebra which has the following form:
$$0\rightarrow S(-12) \rightarrow S(-11) \oplus S(-10) \oplus S(-9) \rightarrow S^{3}(-6)\rightarrow S.$$
Since $(d_{1},d_{2},d_{3}) = (3,4,5)$, $d_{1}+d_{2}=7$, and $d_{3}> d_{2}$, the arrangement $\mathcal{DL}$ is plus-one generated.
\end{proof}

This section is motivated by the so-called addition technique, i.e., starting with a reduced plane curve we add particularly chosen lines in such a way that the resulting curve is free. This methods was studied recently in the setting of inflectional tangent lines (aka hyperosculating lines) and smooth plane curves is \cite{addline}. As a starting point, let us recall \cite[Corollary 1.6]{addline} in the setting of our paper.
\begin{proposition}
\label{Dy1}
Let $\mathcal{C}$ be an arrangement consisting of the Dyck quartic and its $4$ hyperosculating lines such that these lines intersect at one ordinary quadruple point. Then $\mathcal{C}$ is one of three arrangements given by the following defining polynomials:
\begin{itemize}
\item $Q_{1}(x,y,z) = (x^4 + y^4 )\cdot(x^4 + y^4 + z^4 )$,
\item $Q_{2}(x,y,z) = (y^4 + z^4 )\cdot(x^4 + y^4 + z^4 )$,
\item $Q_{3}(x,y,z) = (x^4 + z^4 )\cdot(x^4 + y^4 + z^4 )$.
\end{itemize}
Furthermore, $\mathcal{C}$ is free.
\end{proposition}
\begin{proof}
The first part of our claim follows from the following observation. If we consider the arrangement of $12$ hyperosculating lines $\mathcal{H}$ to the Dyck quartic, given by $Q(x,y,z)=(x^4 + y^4 )\cdot (y^4 + z^4 ) \cdot (x^4 + z^4 )$, then $\mathcal{H}$ has $48$ double intersections and exactly $3$ quadruple intersections. Having already observed this, it is easy to find all possible equations. The second claim follows from \cite[Theorem 1.5]{addline}.
\end{proof}
We should point out here that the situation when four hyperosculating lines to a smooth plane quartic curve meet in one point is rare. For example, in the case of the Klein quartic we have no hyperosculating lines at all, and in the case of the Komiya--Kuribayashi quartic hyperosculating lines intersect only at double points.

Continuing our discussion with the Dyck curve and its bitangents, we present a new type of free curves.
\begin{proposition}
Let $\mathcal{C}$ be an arrangement consisting of the Dyck quartic and its five hyperosculating lines such that four of them intersect at one ordinary quadruple point. Then the arrangement $\mathcal{C}$ is free. 
\end{proposition}
\begin{proof}
First we see that there are $24$ of such arrangements. Let us describe the first $8$ arrangements using their defining polynomials, namely
$$H_{i}^{1}(x,y,z) = Q_{1}(x,y,z)\cdot \ell_{i}(x,y,z),$$
where $\ell_{i} (x,y,z)$ is one of the $8$ linear forms that are given by the linear factors of $(y^4 + z^4)\cdot(x^4 + z^4)$. In the same way we obtain arrangements given by polynomials $H^{2}_{i}$ and $H^{3}_{i}$. In each case we have the same total Tjurina number that is equal to $48$, i.e.,  we have five singular points of type $A_{7}$, one ordinary quadruple intersection, and four nodes. Using \verb}SINGULAR}, we can check that for each arrangement $\mathcal{C}$ the minimal degree of the Jacobian relations $d_{1}$ is equal to $4$. Then
$$48 = d_{1}^{2} - d_{1}(d-1) + (d-1)^2 = \tau(\mathcal{C}) = 48,$$
hence $\mathcal{C}$ is free.
\end{proof}
We now present examples of nearly-free arrangements of degree $10$.
We continue with the setting of the Dyck curve and consider now arrangements given by the following defining polynomials:
\begin{equation*}
\label{gijk}
G_{1}^{i,j}(x,y,z) = Q_{1}(x,y,z)\cdot \ell_{i}(x,y,z)\cdot\ell_{j}(x,y,z),
\end{equation*}
where $\ell_{i}(x,y,z)$ and $\ell_{j}(x,y,z)$ are mutually distinct linear forms which are the linear factors of $(x^4 + z^4)\cdot(y^{4}+z^{4})$, and $Q_{1}$ is defined as in Proposition \ref{Dy1}. In the same and analogous way, we define arrangements given by $G_{2}^{i,j}$ and $G_{3}^{i,j}$.
\begin{proposition}
Arrangements $\mathcal{C}_{k}^{i,j}$ defined by polynomials $G_{k}^{i,j}$ are nearly-free.
\end{proposition}
\begin{proof}
Observe that in each case the total Tjurina number is equal to $60$, i.e., we have $6$ singular points of type $A_{7}$, one ordinary quadruple point, and $9$ nodes. Next, in each case we compute the minimal degree of the Jacobian relations $d_{1}$ that is equal to $5$. Then
$$61 = d_{1}^{2} - d_{1}(d-1) + (d-1)^{2} = \tau(\mathcal{C}_{k}^{i,j})+1 = 60+1 = 61,$$
so $\mathcal{C}_{k}^{i,j}$ are nearly-free.
\end{proof}
Finally, let us present three examples of nearly-free arrangements of degree $12$.
\begin{proposition}
Let us consider the following plane curves:
\begin{itemize}
\item $C_{1} = V((x^{4}+y^{4}+z^{4})\cdot(x^{4}+y^{4})\cdot(y^{4}+z^{4})),$
\item $C_{2} = V((x^{4}+y^{4}+z^{4})\cdot(x^{4}+y^{4})\cdot(x^{4}+z^{4})),$
\item $C_{3} = V((x^{4}+y^{4}+z^{4})\cdot(y^{4}+z^{4})\cdot(x^{4}+z^{4})).$
\end{itemize}
Then curves $C_{1}, C_{2}, C_{3}$ are nearly-free with exponents $(d_{1},d_{2},d_{3}) = (5,7,7)$.
\end{proposition}
\begin{proof}
We leave the proof to the reader.
\end{proof}
To conclude this section, we focus on the Klein quartic and its bitangents. We can construct a bunch of plus-one generated curves which turn out to have a very important meaning.
\begin{proposition}
An arrangement $\mathcal{QK}$ consisting of the Klein quartic and its four bitangent lines that intersect at one quadruple point is plus-one generated with the exponents $(d_{1},d_{2},d_{3}) = (4,4,7)$.
\end{proposition}
\begin{proof}
We have $21$ such arrangements. They have the same singularities, namely eight tacnodes and one ordinary quadruple point. We compute the minimal free resolution of the Milnor algebra which has the following form:
$$0\rightarrow S(-15) \rightarrow S(-14) \oplus S^{2}(-11) \rightarrow S^{3}(-7)\rightarrow S.$$
Since $(d_{1},d_{2},d_{3}) = (4,4,7)$, $d_{1}+d_{2}=8$ and $d_{3}>d_{2}$, the arrangement is plus-one generated.
\end{proof}
Now we need to explain why these arrangements are interesting. Recently, Dimca and Sticlaru have studied homological properties of $m$-syzygy curves. Among many things, they proved the following results -- see \cite[Theorem 2.4]{dimsti} and \cite[Corollary 5.3]{dimsti}.
\begin{theorem} Let $C \, : \, f=0$ be an $m$-syzygy curve of degree $d$ with $1 \leq d_{1} \leq ... \leq d_{m}$ and $m \geq 3$. Then $d_{1}+d_{2} \geq d$ and $d_{1}\leq d_{2} \leq d_{3} \leq d-1$.    
\end{theorem}
\begin{proposition}
Let $C \, : \, f=0$ be a $3$-syzygy curve of degree $d\geq 3$. If all irreducible components $C_{i}$ of $C$ are rational and $C$ is not a plus-one generated curve, then $d_{3}\leq d-2$.    
\end{proposition}
In the light of the above results, the above Klein's arrangements are extreme cases, both in terms of achieving the upper bound on $d_{3}$ as non-rational arrangements, and in terms of showing that all the assumptions in the above proposition are optimal.

\section*{Acknowledgments}
Piotr Pokora would like to thank Xavier Roulleau for his help with numerical computations using MAGMA and for useful remarks, to Igor Dolgachev for his helpful conversations that allowed to understand quadruple intersections of bitangents, and to Ivan Cheltsov for discussions on the subject of this paper in one of Krak\'ow's Starbucks. Finally, we would like to thank the anonymous reviewer for all the comments and suggestions that allowed us to improve the paper.

Piotr Pokora is partially supported by The Excellent Small Working Groups Programme \textbf{DNWZ.711/IDUB/ESWG/2023/01/00002} at the University of the National Education Commission Krakow.

\vskip 0.5 cm
\bigskip
Marek Janasz,
Department of Mathematics,
University of the National Education Commission Krakow,
Podchor\c a\.zych 2,
PL-30-084 Krak\'ow, Poland. \\
\nopagebreak
\textit{E-mail address:} \texttt{marek.janasz@uken.krakow.pl}

\bigskip
Piotr Pokora,
Department of Mathematics,
University of the National Education Commission Krakow,
Podchor\c a\.zych 2,
PL-30-084 Krak\'ow, Poland. \\
\nopagebreak
\textit{E-mail address:} \texttt{piotrpkr@gmail.com, piotr.pokora@uken.krakow.pl}

\bigskip
Marcin Zieli\'nski,
Department of Mathematics,
University of the National Education Commission Krakow,
Podchor\c a\.zych 2,
PL-30-084 Krak\'ow, Poland. \\
\nopagebreak
\textit{E-mail address:} \texttt{marcin.zielinski@uken.krakow.pl}
\bigskip
\end{document}